\documentclass[hidelinks]{article}
\usepackage{graphicx} 
\usepackage{hyperref}
\usepackage{amsthm}
\usepackage{geometry}
\usepackage{amssymb}
\usepackage{amsmath}
\usepackage{enumitem}
\usepackage{color}
\definecolor{purple}{rgb}{0.5,0,0.5}
\usepackage{lmodern}
\usepackage[english]{babel}

\newtheorem{theorem}{Theorem}[section]
\newtheorem{lemma}[theorem]{Lemma}

\newtheorem{claim}[theorem]{Claim}

\theoremstyle{definition}

\newcommand{\claimproofstart}[1][Proof]{\begin{proof}[#1]
\renewcommand{\qedsymbol}{$\boxdot$}}
\newcommand\claimproofend{
\end{proof}
\renewcommand{\qedsymbol}{$\square$}}

\newcommand{\cA}{\mathcal{A}}
\newcommand{\cC}{\mathcal{C}}

\newcommand{\cH}{\mathcal{H}}
\newcommand{\cQ}{\mathcal{Q}}
\newcommand{\cG}{\mathcal{G}}
\newcommand{\cF}{\mathcal{F}}
\newcommand{\cE}{\mathcal{E}}
\newcommand{\cP}{\mathcal{P}}

\newcommand{\Var}{\mathrm{Var}}

\title{Simplified dense model theorem for random discrete structures}

\author{Rajko Nenadov \footnote{School of Mathematics and Statistics, University of Canterbury, New Zealand. Research supported by the New Zealand Marsden Fund. Email: rajko.nenadov@canterbury.ac.nz}} 
\date{}

\begin{document}

\maketitle

\begin{abstract}
    Dense model theorems of Conlon and Gowers associate to each subset of a sparse random structure a dense set with approximately the same size and configuration counts after rescaling. This allows dense extremal results to be transferred to the random setting. Allen, B\"ottcher, Lada, and Mergoni Cecchelli recently generalised this approach and established a dense model theorem in an abstract hypergraph framework, with hypotheses that are readily verified in many applications. Both proofs use the separating hyperplane theorem to reduce the existence of dense models to an anti-correlation estimate. We give a simplified proof of a version of this theorem, inspired by the method of hypergraph containers. Our proof constructs dense models from small fingerprints, combining sampling with the R\"odl--Ruci\'nski deletion method.
\end{abstract}

\section{Introduction}

Transferring extremal results to the sparse random setting is a central topic in probabilistic combinatorics. The classical Erd\H{o}s--Stone theorem \cite{erdos46turan}, extending Tur\'an's theorem, states that, for every fixed graph $H$ with at least one edge and every $\varepsilon>0$, if $G \subseteq K_n$ satisfies
$$
    e(G) \ge \left(1 - \frac{1}{\chi(H) - 1} + \varepsilon \right) e(K_n),
$$
and $n$ is sufficiently large, then $G$ contains $H$. A natural random graph analogue asks for which values of $p$ the graph $\Gamma \sim G(n,p)$ has the same relative property with high probability, namely that every $G \subseteq \Gamma$ with
$$
    e(G) \ge \left(1 - \frac{1}{\chi(H) - 1} + \varepsilon \right) e(\Gamma)
$$
contains $H$. Such a property is usually referred to as the \emph{relative Tur\'an property}.

This and related extremal and Ramsey questions have shaped the development of probabilistic combinatorics. A first major breakthrough was due to R\"odl and Ruci\'nski \cite{rodl1995threshold}, who determined thresholds for Ramsey properties of random graphs. Friedgut, R\"odl, and Schacht \cite{friedgut10random} subsequently proved a general result for Ramsey properties of random hypergraphs and other random discrete structures. Random analogues of the Erd\H{o}s--Stone theorem were established by Conlon and Gowers \cite{conlon16turan} and Schacht \cite{schacht16turan}, using substantially different methods. A third approach for tackling an even wider range of problems, the \emph{hypergraph containers} method, was developed by Balogh, Morris, and Samotij \cite{balogh2015independent} and, independently, by Saxton and Thomason \cite{saxton2015hypergraph}. All these methods aim to \emph{transfer} an existing extremal result for dense structures to sparse random ones, using the dense result as a black box. A related approach uses Szemer\'edi's regularity lemma and the sparse counting result, the celebrated \emph{K\L R conjecture} \cite{conlon2014klr,kohayakawa1997k,nenadov22klr} (first proved using the aforementioned hypergraph containers \cite{balogh2015independent,saxton2015hypergraph}), allowing one to mimic regularity-based proofs in subgraphs of random graphs. Let us say a few words about the three transference approaches.

\begin{itemize}
    \item The approach of R\"odl and Ruci\'nski \cite{rodl1995threshold}, later refined by Schacht \cite{schacht16turan}, builds copies of $H$ inductively on the number of edges. The dense result guarantees sufficiently many potential extensions, while second-moment bounds control their distribution. To obtain sufficiently strong probability estimates for the required union bounds, R\"odl and Ruci\'nski developed the \emph{deletion method}, which also plays an important role in our argument. 
    
    \item The \emph{dense-model} approach of Conlon and Gowers \cite{conlon16turan} is based on the work of Green and Tao \cite{green08primes}. For a fixed graph $H$ and sufficiently large $p$, they show that $\Gamma \sim G(n,p)$ typically has the following property: for every $G \subseteq \Gamma$, there exists $D \subseteq K_n$ such that $e(G) \approx e(D)p$ and $N_H(G) \approx N_H(D)p^{e(H)}$, where $N_H(G)$ counts copies of $H$ in $G$. Here the approximation errors are additive, measured on the scales $e(K_n) p$ and $N_H(K_n) p^{e(H)}$, respectively. Such a graph $D$ is called a \emph{dense model} of $G$. Applying the dense supersaturation theorem to $D$ then gives the relative Tur\'an property. 
    
    The proof first introduces a norm designed to control discrepancies in edge and $H$-copy counts. This norm is defined using functions that encode normalised counts of extensions of an edge to a copy of $H$. The separating hyperplane theorem, together with polynomial approximation, reduces the existence of a dense model to showing that products of a bounded number of these functions have small correlation with the difference between the rescaled random graph and the complete graph (for a detailed description of the argument, see \cite[Section 4]{allen26dense}).

    The approach of Conlon and Gowers has recently been generalised by Allen, B\"ottcher, Lada, and Mergoni Cecchelli \cite{allen26dense}. In particular, \cite[Theorem 2.4]{allen26dense} gives a general dense-model framework in the language of hypergraphs.

    \item The \emph{container} approach constructs a small collection $\cC$ of subgraphs of $K_n$ such that every $H$-free graph $I \subseteq K_n$ is contained in some $C \in \cC$, and every $C \in \cC$ contains few copies of $H$. Supersaturation bounds the number of edges in each container, and a union bound, which crucially relies on $\cC$ being small, bounds the number of edges in any $H$-free subgraph of $G(n,p)$. Such a family $\cC$ is constructed as follows: Given an $H$-free $I \subseteq K_n$, one first extracts a small representative subgraph $F \subseteq I$, called a \emph{fingerprint}, and constructs $C$ based on the fact that any edge of $K_n$ that would complete a copy of $H$ using edges of $F$ cannot belong to $I$. The size of the fingerprints controls the number of possible containers. Their containment in $I$ is crucial for the union bound in the random setting.
\end{itemize}

Each of the three approaches has its advantages. In particular, dense model theorems allow one to transfer extremal inequalities, such as those given by the clique density theorem \cite{reiher16clique}. The three approaches use substantially different proof mechanisms, so it is natural to ask whether they share a common underlying idea. We connect them by constructing dense models from small fingerprints, using second-moment bounds to control the distribution of hyperedges. More precisely, building on the author's probabilistic version of the hypergraph container theorem \cite{nenadov24prob}, we combine a sampling-based fingerprint construction with the R\"odl--Ruci\'nski deletion method to give an alternative proof of a version of \cite[Theorem 2.4]{allen26dense}. The fingerprint method thus extends from controlling independent sets to preserving configuration counts and intersections with prescribed sets.

\paragraph{Hypergraph setup.}
The hypergraph container framework encodes a problem as one about independent sets in an appropriately defined hypergraph. The main idea is to cover all independent sets by a small family of almost independent sets. We use the same idea, but instead of covering independent sets, we construct dense models that work for any given subset.

For a finite set $V$, let $V_p$ denote the random subset obtained by including each vertex independently with probability $p$. For a hypergraph $\cH$ and $S\subseteq V(\cH)$, let $\cH[S]$ be the induced hypergraph and write $e_\cH(S)=e(\cH[S])$. Given $q \in (0, 1]$, set
$$
    \mu_\cH(q) = \sum_{h \in \cH} q^{|h|} \quad \text{ and } \quad \Delta_\cH(q) = \sum_{\substack{h, h' \in \cH \\ h \cap h' \neq \emptyset}} q^{|h \cup h'|},
$$
where the second sum is over ordered pairs $(h,h')$, including the diagonal pairs $h=h'$. 

We prove the following version of \cite[Theorem 2.4]{allen26dense}.

\begin{theorem} \label{thm:main}
    For every integer $k \ge 2$, $\varepsilon > 0$ and $K \ge 1$ there exist $c,\delta > 0$ and $T > 1$ such that the following holds. Let $\cH$ be a $k$-uniform hypergraph with $N \ge 1$ vertices, and suppose $q \in [1/N, 1]$ is such that
    \begin{equation} \label{eq:Delta_hypothesis}
        \Delta_\cH(q) \le K \frac{\mu_\cH(q)^2}{N q}.
    \end{equation}
    Let $p \in (Tq, 1]$, and let $\cA$ be a family of subsets of $V(\cH)$ with $|\cA| \le 2^{\delta Np}$. For $X\sim V(\cH)_p$, with probability at least 
    $$
        1-e^{-cNp}
    $$
    the following holds: For every $Y \subseteq X' \subseteq X$ with $|X'| \ge (1 - \delta)Np$ and $e_\cH(X') \le ( 1 + \varepsilon/4) \mu_\cH(p)$, there exists $C\subseteq V(\cH)$ such that
    $$
        \Bigl| |Y \cap A| - |C \cap A|p \Bigr| \le \varepsilon Np \quad \text{ for every $A \in \cA$}
    $$
    and
    $$
        \Bigl| e_\cH(Y) - e_\cH(C) p^k \Bigr| \le \varepsilon e(\cH) p^k.
    $$ 
\end{theorem}

The moment bound \eqref{eq:Delta_hypothesis} has a natural necessary consequence. Since the diagonal terms give $\Delta_\cH(q)\ge\mu_\cH(q)$, it implies $\mu_\cH(q)\ge Nq/K$ whenever $e(\cH)>0$, and hence $\mu_\cH(p)\ge Np/K$ for $p\ge q$. If instead $\mu_\cH(p)=o(Np)$ and $Np\to\infty$, then typically one could delete one vertex from each edge of $\cH[X]$ to obtain an independent set $Y\subseteq X$ of size $(1-o(1))Np$. Its dense model would have to be nearly independent and contain almost all the vertices, which need not be possible.

The upper bound on $e_{\cH}(X')$ is necessary up to the error constant: taking $Y=X'$ would force $e_{\cH}(X')\le (1+\varepsilon)\mu_{\cH}(p)$. We remark that Lemma \ref{lemma:rodl_rucinski} and Chernoff’s inequality show that a set $X'$ satisfying both hypotheses can be obtained by deleting at most $\delta Np/2$ vertices from $X$, with failure probability $\exp{-\Omega(Np)}$. If $Np\to\infty$, the bound \eqref{eq:Delta_hypothesis} and Chebyshev’s inequality allow us to take $X'=X$ with high probability.

To understand the implications of Theorem \ref{thm:main}, a reader might consider the case $\cA = \{V(\cH)\}$, which suffices for many applications. The freedom to further control how $Y$ and $C$ intersect certain subsets of $V(\cH)$ allows one to preserve additional structure. An application where this is crucial is given in Section~\ref{sec:example}. Let us now make a connection between Theorem \ref{thm:main} and the aforementioned dense model theorem of Conlon and Gowers. Let $H$ be a fixed graph with at least two edges and no isolated vertices, and consider the hypergraph $\cH$ where vertices correspond to edges of $K_n$, the complete graph with $n$ vertices (hence $N = \binom{n}{2}$), and a set of $e(H)$ vertices forms a hyperedge if and only if the corresponding edges in $K_n$ form a copy of $H$. Suppose $p$ is suitably large. The random set $X \sim V(\cH)_p$ corresponds to $\Gamma \sim G(n,p)$ and $e_\cH(X)$ counts the number of copies of $H$ in $\Gamma$. The quantity $\mu_\cH(p)$ then denotes the expected number of copies of $H$. On the typical event that $|X|\ge(1-\delta)Np$ and $e_\cH(X)\le(1+\varepsilon/4)\mu_\cH(p)$, for any $G \subseteq \Gamma$ and the associated set $Y \subseteq X$, Theorem \ref{thm:main} provides us with a set $C \subseteq V(\cH)$ such that $|Y| \approx |C| p$ and $e_\cH(Y) \approx e_\cH(C) p^{e(H)}$, that is, $e(G) \approx e(D) p$ and $N_H(G) \approx N_H(D) p^{e(H)}$, where $D \subseteq K_n$ is the graph corresponding to the vertices in $C$. The graph $D$ is thus a dense model of $G$.

\paragraph{Main idea.} Suppose, for simplicity, that $X$ itself satisfies the hypotheses imposed on $X'$. Given $Y\subseteq X$, we seek a set $C\subseteq V(\cH)$ such that $|Y|\approx|C\cap X|$ and $e_\cH(Y)\approx e_\cH(C\cap X)$. For a fixed $C$, chosen before sampling $X$, Chernoff's inequality gives $|C\cap X|\approx p|C|$. Janson's inequality gives lower bounds for the counts in both $\cH[C]$ and $\cH\setminus\cH[C]$. Together with an upper bound on $e_\cH(X)$, these yield the required two-sided approximation to $e_\cH(C) p^k$. Under the $\Delta_\cH(q)$ assumption, the Chernoff and Janson failures have probability $\exp\{-\Omega(Np)\}$, so a union bound permits a family of $2^{o(Np)}$ possible models fixed before sampling $X$. In the proof, we implement this idea through small fingerprints, conditioning on their containment in $X$ and accounting for the probability of that containment.

Unlike most hypergraph container arguments, which apply to every independent set, our argument requires $\cH[X]$ to have well-distributed edges, as quantified in Section \ref{sec:deletion}. Using the structure of $\cH[X]$, we show how to construct $C$ by examining small random samples of $Y$ and $X$. Both of these ideas -- restricting to sets with some special structure and using sampling to construct the container -- have been used in \cite{nenadov24prob}. What is different here is what we try to achieve with $C$.

A further simplification is to pass to a transversal $t$-partite subhypergraph $\cG\subseteq\cH$, where $t$ is sufficiently large in terms of the error parameter. This allows us to construct the model one part at a time.

\paragraph{Structure of the paper.} In Section \ref{sec:deletion}, we use the R\"odl--Ruci\'nski deletion method to show that a typical set $X \sim V(\cH)_p$ contains a well-structured large subset (Lemma \ref{lemma:deletion}). The need to pass to a subset of $X$ is due to the lack of general upper-tail estimates. Section \ref{sec:fingerprint} contains the main fingerprint-model result (Lemma \ref{lemma:partite}); its overlap-dependent error bound becomes useful after applying the deletion lemma. We combine these results to prove Theorem \ref{thm:main} in Section \ref{sec:proof_main}. Finally, Section \ref{sec:example} illustrates the theorem by deriving the two-sided counting lemma for subgraphs of random graphs. 

\paragraph{Use of AI.} Initial prompts to ChatGPT Astra were based on combining ideas from \cite{nenadov24prob} and the currently unpublished paper of the author on the strong counting K\L R conjecture, to prove a statement along the lines of \cite[Theorem 2.4]{allen26dense}. Astra then suggested that an appropriate notion of a ``container'' in the case of dense model theorems for random discrete structures is a set $C$ such that $|Y| \approx |C \cap X|$ and $e_\cH(Y) \approx e_\cH(C \cap X)$. AI assistance was also used in writing the exposition, particularly in setting up the constants. The author takes responsibility for the mathematical claims and references.

\section{Deletion lemma}
\label{sec:deletion}

The following lemma is an application of the R\"odl--Ruci\'nski deletion method.

\begin{lemma} \label{lemma:rodl_rucinski}
    Let $(S_a)_{a \in A}$ be a finite labelled family of nonempty subsets of $V$, each of size at most $r\ge1$. For every integer $d \ge 0$ and $B > 1$, the random set $X \sim V_p$ satisfies
    $$
        \Pr\bigl(Z(X \setminus R) > B \; \mathbb{E} Z(X) \text{ for every } R \subseteq X,\ |R| \le d\bigr) \le B^{-(\lfloor d/r\rfloor+1)},
    $$
    where $Z(U)=|\{a:S_a\subseteq U\}|$ for $U\subseteq V$.
\end{lemma}
\begin{proof}
    Set $\mu = \mathbb{E} Z(X)$. If $\mu=0$, the event is impossible. Otherwise, set $s=\lfloor d/r\rfloor+1$ and suppose $X$ is such that $Z(X \setminus R) > B\mu$ for every $R \subseteq X$ with $|R|\le d$. There are more than $(B\mu)^s$ ordered sequences of $s$ pairwise disjoint labelled sets $S_a$ contained in $X$. Indeed, the union of any $j<s$ chosen sets has size at most $jr\le d$, so deleting this union leaves more than $B\mu$ choices for the next set. The expected number of these sequences is at most $\mu^s$, since containment events for disjoint supports are independent. Markov's inequality gives the desired bound.
\end{proof}

For a hypergraph $\cG$ and a set $S \subseteq V(\cG)$, set
\begin{equation} \label{eq:error_bound}
    D_\cG(S) = \bigl| \left\{ (h, h') \; : \; h, h' \in \cG[S] \text{ and } |h \cap h'| \ge 2  \right\} \bigr|.
\end{equation}
Here pairs are ordered and diagonal pairs (that is, $h = h'$) are included. The next lemma makes this overlap count small by deleting a few vertices.
\begin{lemma} \label{lemma:deletion}
    For every $k \in \mathbb{N}$, $\alpha, \gamma > 0$ and $K \ge 1$ there exists $T > 1$ such that the following holds. Let $\cH$ be a $k$-uniform hypergraph with $N \ge 1$ vertices, and suppose $q \in (0, 1]$ is such that
    $$
        \Delta_\cH(q) \le K \frac{\mu_\cH(q)^2}{N q}.
    $$
    Let $p \in (Tq, 1]$ and $X \sim V(\cH)_p$. Then, with probability at least $1 - \exp\{-Np\}$, there exists $R \subseteq X$ of size $|R| \le \alpha Np$ such that $X' := X \setminus R$ satisfies
    $$
        D_\cH(X') \le \gamma \Delta_\cH(p).
    $$
\end{lemma}
\begin{proof}
    Choose $T=1+K\gamma^{-1}\exp(2k/\alpha)$. If $k=1$, then $D_\cH(X)=0$ and we take $R=\emptyset$. Hence assume $k\ge2$. Note that
    $$
        \mathbb{E}D_\cH(X) \le(p/q)^{2k-2} \Delta_\cH(q) \le K\frac qp \; \frac{\mu_\cH(p)^2}{Np}.
    $$
    By the Cauchy--Schwarz inequality, we have
    $$
        \Delta_\cH(p)\ge\frac{p^{2k-1}}{k} \sum_{v\in V(\cH)}\deg_\cH(v)^2 \ge \frac{\mu_\cH(p)^2}{Np}.
    $$
    Therefore, since $p \ge Tq$, we conclude $\mathbb{E}D_\cH(X) \le K \Delta_\cH(p) / T$.

    Apply Lemma~\ref{lemma:rodl_rucinski} to the labelled sets $h\cup h'$,
    one for each ordered pair $h,h'\in\cH$ with $|h\cap h'|\ge2$.
    Their counting function is $D_\cH$, and each has size at most $r = 2k - 2$.
    Set $d=\lfloor\alpha Np\rfloor$ and $B=\exp(2k/\alpha)$.
    Since $\lfloor d/r\rfloor+1>\alpha Np/(2k)$,
    with probability at least $1 - \exp\{-Np\}$ there is
    $R\subseteq X$, $|R|\le\alpha Np$, such that
    $$
        D_\cH(X\setminus R)\le B\; \mathbb{E}D_\cH(X)
        \le \frac{B \; K}{T}\Delta_\cH(p)\le\gamma\Delta_\cH(p).
    $$
\end{proof}

\section{Fingerprint-model lemma} 
\label{sec:fingerprint}

The following is the key lemma in the proof of Theorem \ref{thm:main}. The conclusion of the lemma is that there exists a set $C$ with properties that resemble those in Theorem \ref{thm:main}. For a fixed set $C$ and $X \sim V_p$, we expect $|C \cap X \cap A| \approx |C \cap A|p$ and $e_\cG(C \cap X) \approx e_\cG(C) p^k$. The fact that the model set $C$ depends only on a small fingerprint will allow us to use a union bound to conclude that these approximations hold for every possible $C$, thus giving the conclusion of Theorem \ref{thm:main}.

\begin{lemma} \label{lemma:partite}
    For every $k,t\in\mathbb N$ with $2\le k\le t$ and $\alpha,\varepsilon\in(0,1]$ there exist $c,Q>0$ and $h\in\mathbb N$ such that the following holds. Let $\cG$ be a $k$-uniform hypergraph on $V=V_1\sqcup\cdots\sqcup V_t$, with each edge meeting each part at most once. Let $\emptyset\ne Y\subseteq X\subseteq V$, $|X| \ge 1/c$, and let $\cA = \{A_i\}_{i \in I}$ be a nonempty family of subsets of $V$ with $|I|\le2^{c|X|}$. 
    
    There exist $\mathbf F_Y \in (2^Y)^{2t}$, $\mathbf{F}_X \in (2^X)^{2t}$, and $\mathbf{J} = (J_1,\ldots,J_t) \in (2^I)^t$, such that
    $$
        \sum_{X' \in\mathbf F_X} |X'| + \sum_{Y' \in \mathbf F_Y}|Y'| \le \alpha|X| \quad \text{ and } \quad \sum_{J_i \in \mathbf{J}} |J_i| \le h t,
    $$
    and a set $C=C(\mathbf F_X, \mathbf F_Y, \mathbf{J})\subseteq V$, such that
    $$
        \Bigl||Y\cap A|-|C\cap X\cap A|\Bigr|\le\varepsilon|X|
        \quad\text{for every $A\in\cA$}
    $$
    and
    $$
        \Bigl|e_\cG(Y)-e_\cG(C\cap X)\Bigr|\le Q\sqrt{|X|D_\cG(X)}.
    $$
\end{lemma}

We emphasise that the set $C$ depends on the ambient hypergraph $\cG$ and the family $\cA$, and furthermore on tuples $\mathbf{F}_X$, $\mathbf{F}_Y$, and $\mathbf{J}$. It does not depend on the whole sets $X$ and $Y$, nor on the sets $A_j$ whose indices lie outside $\bigcup_{i=1}^t J_i$. The error parameter $D_\cG(X)$ is defined in \eqref{eq:error_bound}.

For the rest of this section, whenever we consider a particular finite set, we assume that there is a fixed total ordering on it. For a finite set $V$, a subset $U \subseteq V$ and a function $\phi: V \to \mathbb{R}$, set
$$
    \phi(U) = \sum_{u \in U} \phi(u).
$$
The following lemma is the main building block in the proof of Lemma \ref{lemma:partite}.

\begin{lemma} \label{lemma:approx_full}
    For every $\alpha, \varepsilon > 0$ there exist $c > 0$ and $B, h \in \mathbb{N}$ such that the following holds. Let $V$ be a finite ordered set, $\phi : V \to [0,\infty)$, and $\emptyset \neq Y \subseteq X \subseteq V$ with $|X| \ge 1/c$. Let $\cA = \{A_i\}_{i \in I}$ be a family of subsets of $V$ with $|I| \le 2^{c |X|}$. There exist $X' \subseteq X$, $Y' \subseteq Y$, and $J \subseteq I$, such that 
    $$
        |X'| + |Y'| \le \alpha |X| \quad \text{ and } \quad |J| \le h,
    $$
    and a set $C = C(\phi, X', Y', J) \subseteq V$, such that
    $$
        \Bigl| |Y \cap A| - |C \cap X \cap A| \Bigr| \le \varepsilon |X| \quad \text{for every $A \in \cA$}
    $$
    and
    $$
        \Bigl| \phi(Y) - \phi(C \cap X) \Bigr| \le B \left( \sum_{x \in X} \phi(x)^2 \right)^{1/2}.
    $$
\end{lemma}
\begin{proof}
    By decreasing $\alpha$ and $\varepsilon$, we may assume that both lie in $(0,1]$. Set
    $$
        h=\lceil4/\varepsilon^2\rceil,\qquad
        s=\left\lceil\frac{4(1+2^h)}{\alpha}\right\rceil,\qquad
        c=\min\left\{\frac{\varepsilon^2}{8s\log2},\frac{\alpha}{4\cdot2^h}\right\},\qquad
        B=\left\lceil\max\{2s,c^{-1/2}\}\right\rceil.
    $$
    Define the order $\prec$ on $V$ by non-decreasing $\phi$, tie-breaking according to the given total order on $V$. Let $x_1 \prec \ldots \prec x_n$ denote the induced ordering of $X$, where $n = |X|$. Set
    $$
        X' := \{ x_{j s} \; : \; 1 \le j \le \lfloor n / s \rfloor \}.
    $$
    Cutting immediately after these points gives the partition $\cP=\{P_1,\ldots,P_m\}$ of $X$, where $m=\lceil n/s\rceil$ and
    $$
        P_j=\{x_\ell:(j-1)s<\ell\le\min\{js,n\}\}.
    $$
    Every part has size at most $s$ and $m\le n/s+1$. We now refine $\cP$ as follows: As long as there exists $A_i\in\cA$ such that
    \begin{equation} \label{eq:refine}
        \left| |Y\cap A_i|-\sum_{P\in\cP}y(P)|P\cap A_i|\right|\ge\varepsilon n/2,
        \qquad \text{where } y(P)=\frac{|Y\cap P|}{|P|},
    \end{equation}
    refine $\cP$ by $A_i$, splitting each part $P \in \cP$ into $P \cap A_i$ and $P \setminus A_i$ (discarding empty parts). The next claim shows that this process finishes after at most $h$ iterations.

    \begin{claim}
        Let $J\subseteq I$ be the indices used in this process. Then $|J|\le h$.
    \end{claim}
    \claimproofstart
        The quantity $\cE(\cP) := \sum_{P\in\cP}|P|y(P)^2$, which we shall refer to as the \emph{energy} of $\cP$, lies in $[0,n]$. If $\cP'$ is obtained by refining $\cP$ by $A_i$, its increase is
        $$
            \cE(\cP') - \cE(\cP) = \sum_{P\in\cP}\ \sum_{\substack{P' \in \cP'\\P' \subseteq P}}
                |P'|\bigl(y(P')-y(P)\bigr)^2 \ge \sum_{P \in \cP} \sum_{\substack{P' \in \cP'\\P' \subseteq P}} |P' \cap A_i| (y(P') - y(P))^2.
        $$
        We also have
        $$
            \varepsilon n / 2 \le \Bigl| |Y\cap A_i|-\sum_{P\in\cP}y(P)|P\cap A_i| \Bigr|           
            = \Bigl| \sum_{P \in \cP} \sum_{\substack{P' \in \cP' \\ P' \subseteq P}} (y(P') - y(P))|P' \cap A_i| \Bigr|.
        $$
        The Cauchy--Schwarz inequality implies that the increase is thus at least $\varepsilon^2n/4$. As the energy is bounded by $n$, there are at most $4/\varepsilon^2\le h$ refinements. 
        
        Note that a set $A_i$ cannot satisfy \eqref{eq:refine} more than once, since refining $\cP$ by $A_i$ again would not change it and hence the energy would not change -- contradicting the fact that it should increase by at least $\varepsilon^2 n / 4$.
    \claimproofend

    Let $\cP$ now denote the final partition. It has at most $2^h m$ parts, each of size at most $s$, and the discrepancy in \eqref{eq:refine} is less than $\varepsilon n/2$ for every $A \in \cA$. For each $P\in\cP$, independently, choose $x_P\in P$ uniformly at random, and add it to $Y'$ if $x_P \in Y$. Set
    $$
        C'=\bigcup_{\substack{P\in\cP\\x_P\in Y'}}P.
    $$
    A particular $P \in \cP$ is included in $C'$ with probability $y(P)$. As $\sum_{P\in\cP}|P|^2\le s n$,  Hoeffding's inequality gives
    $$
        \Pr\left(\bigl||C'\cap A|-\mathbb E|C'\cap A|\bigr| 
        \le \varepsilon n/2\right) \ge 1 - 2\exp\left(-\frac{\varepsilon^2n}{2s}\right).
    $$
    By the union bound over $\cA$, with probability at least, say, $3/4$, this event holds simultaneously for every $A \in \cA$. Since $\mathbb E|C'\cap A|=\sum_{P\in\cP}y(P)|P\cap A|$, the stopping condition  \eqref{eq:refine} and the triangle inequality imply
    \begin{equation} \label{eq:C_prime_occupancy}
        \bigl| |Y \cap A_i| - |C' \cap A_i| \bigr| \le \varepsilon n \quad \text{for every $A_i \in \cA$.}
    \end{equation}

    We now estimate $|\phi(C') - \phi(Y)|$. By the linearity of expectation, we have
    $$
        \mathbb E[\phi(C')]-\phi(Y)  =\sum_{\substack{P\in\mathcal P \\ Y \cap P \neq \emptyset}} \left(\frac{|Y\cap P|}{|P|}\phi(P)-\phi(Y\cap P)\right) 
        =\sum_{P\in\mathcal P}\sum_{x\in Y\cap P} \left(\frac{\phi(P)}{|P|}-\phi(x)\right).
    $$  
    The average $\phi(P)/|P|$ and each $\phi(x)$ both lie between the minimum and maximum of $\phi$ on $P$. Thus
    $$
      \left|\mathbb E[\phi(C')]-\phi(Y)\right| 
      \le\sum_{P\in\mathcal P}|Y\cap P| \left(\max_{x\in P}\phi(x)-\min_{x\in P}\phi(x)\right) 
      \le\sum_{P\in\mathcal P}|P| \left(\max_{x\in P}\phi(x)-\min_{x\in P}\phi(x)\right).
    $$
    Note that the last sum does not increase under refinement. For the original consecutive blocks, the sum of the $\phi$-ranges is at most $\max_{x\in X}\phi(x)-\min_{x\in X}\phi(x)$. Since each block has size at most $s$ and $\phi\ge0$, we conclude
    $$
        \left| \mathbb{E}[\phi(C')] - \phi(Y) \right| \le s \max_{x \in X} \phi(x).
    $$
    Independence of the choices in distinct parts and Cauchy--Schwarz also give
    $$
        \Var\left[ \phi(C') \right] \le \frac{1}{4} \sum_{P\in\cP}\phi(P)^2 \le \frac{s}{4}\sum_{x\in X}\phi(x)^2.
    $$
    Chebyshev's inequality now shows that, with probability at least $3/4$,
    \begin{equation} \label{eq:C_prime_phi}
        |\phi(C')-\phi(Y)| \le B\left(\sum_{x\in X}\phi(x)^2\right)^{1/2}.
    \end{equation}
    Finally, we deterministically have
    $$
        |X'|+|Y'| \le n/s+2^h(n/s+1) \le\alpha n/4+2^h\le\alpha n/2\le\alpha n.
    $$
    There is thus a choice of $Y'$ for which $C'$ satisfies \eqref{eq:C_prime_occupancy} and \eqref{eq:C_prime_phi}.

    We now define $C$ based on $X', Y'$, and $J$ as follows. Partition $V$ into intervals by cutting immediately after each point of $X'$ in the order $\prec$, and refine this partition by every $A_j$ with $j\in J$, in an arbitrary order. Denote the resulting partition by $\cQ$, and let
    $$
        C=\bigcup_{\substack{Q\in\cQ\\Q\cap Y'\ne\emptyset}}Q.
    $$
    Observe that $\{ X \cap Q \; : \; Q \in \cQ, X \cap Q \neq \emptyset\}$ is precisely the partition $\cP$. This further implies $C\cap X=C'$, so $C$ has the desired properties.
\end{proof}

With the previous lemma at hand, we are ready to proceed to the proof of Lemma \ref{lemma:partite}.

\begin{proof}[Proof of Lemma \ref{lemma:partite}]
    Fix an arbitrary ordering of $V$. Let $c,B,h$ be given by Lemma \ref{lemma:approx_full} with parameters $\alpha/(2t)$ and $\varepsilon/t$, and set $L=(16t/\alpha)^{k-1}$ and $Q=tL(8+B)$. For each $i \in [t]$ set $V_{< i} = V_1 \cup \ldots \cup V_{i-1}$ and $V_{>i} = V_{i+1} \cup \ldots \cup V_t$ ($V_{\le i}$ and $V_{\ge i}$ are defined analogously). For $i \in [t+1]$ and $A, B \subseteq V$, define
    $$
        e_{\cG}^{(i)}(A, B) = \Bigl| \{ h \in \cG \; : \; h \cap V_{< i} \subseteq A \text{ and } h \cap V_{\ge i} \subseteq B \} \Bigr|.
    $$    
    
    We iteratively construct $C_{t} \subseteq C_{t-1} \subseteq \ldots \subseteq C_1$, with $C_i \subseteq V_{\ge i}$, such that for every $i \in [t]$ we have
    \begin{equation} \label{eq:induction}
        \begin{gathered}
            \Bigl||Y\cap V_{\ge i}\cap A|-|C_i\cap X\cap A|\Bigr|
            \le\frac{(t-i+1)\varepsilon|X|}{t}\quad \text{ for every } A\in\cA,\\
            \left|e_\cG^{(i+1)}(Y,C_{i+1}\cap X)-e_\cG^{(i)}(Y,C_i\cap X)\right|
            \le\frac{Q\sqrt{|X|D_\cG(X)}}{t},
        \end{gathered}
    \end{equation}            
    where $C_{t+1} = \emptyset$.
    
    Suppose we have already constructed $C_t, \ldots, C_{i+1}$, for some $i \in [t]$, and set $X_i = X \cap V_i$  and $Y_i=Y \cap V_i$. Define $\phi_i: V_i \to \mathbb{N}_0$ as
    $$
        \phi_i(v) = \Bigl| \left\{ h \in \cG \; : \; v \in h,\; h \cap V_{<i} \subseteq Y \text{ and } h \cap V_{>i} \subseteq C_{i+1} \cap X \right\} \Bigr|.
    $$
    Note that $\phi_i$ depends on $C_{i+1}$, $X$ and $Y$. Our first goal is to approximate $\phi_i$ using less information.

    \begin{claim}
        There exist $S_i \subseteq Y$ and $T_i \subseteq X$ such that $|S_i| + |T_i| \le \alpha |X|/(2t)$ and
        \begin{equation} \label{eq:approx_phi}
            \sum_{x \in X_i} |\phi_i(x) - L \phi_i'(x)| \le 4L \sqrt{|X| D_\cG(X)},    
        \end{equation}
        where $\phi_i' : V_i \to \mathbb{N}_0$ is defined as
        $$
            \phi_i'(v) =  \Bigl| \left\{ h \in \cG \; : \; v \in h,\; h \cap V_{<i} \subseteq S_i \text{ and } h \cap V_{>i} \subseteq C_{i+1} \cap T_i \right\} \Bigr|.
        $$
    \end{claim} 
    \claimproofstart
        Set $\theta = \alpha / (16t)$, and let $S_i \sim Y_\theta$ and $T_i \sim X_\theta$ be independent $\theta$-Bernoulli samples. For $x \in X_i$, set $d(x) = L \phi_i'(x) - \phi_i(x)$. Note that $L \mathbb{E} \phi_i'(x) = \phi_i(x)$. Thus
        $$
            \mathbb{E} d(x)^2 = L^2 \mathbb{E} \phi_i'(x)^2 - \phi_i(x)^2.
        $$
        For each $x\in X_i$, let $\cE_x$ denote the edges $h \in \cG$ counted by $\phi_i(x)$. Then
        $$
            \mathbb{E} \phi_i'(x)^2 \le \sum_{(h, h') \in \cE_x \times \cE_x} \theta^{|h \cup h'|-1}.
        $$
        By splitting the sum depending on whether $h \cap h' = \{x\}$ or not, we conclude
        $$
            \mathbb{E}  \sum_{x \in X_i} \phi_i'(x)^2 \le \theta^{2(k - 1)} \sum_{x \in X_i} \phi_i(x)^2 + D_\cG(X).
        $$
        Here each ordered edge pair in the error meets $V_i$ in at most one vertex, so it is counted at most once. Therefore,
        $$
            \mathbb{E} \sum_{x \in X_i} d(x)^2 \le L^2 D_\cG(X).
        $$      
        By the Cauchy--Schwarz inequality, this implies
        $$
            \mathbb{E} \sum_{x \in X_i} |L \phi_i'(x) - \phi_i(x)| \le L \sqrt{|X| D_\cG(X)}.
        $$
        By Markov's inequality, with positive probability we have
        $$
            \sum_{x \in X_i} |L \phi_i'(x) - \phi_i(x)| \le 4L \sqrt{|X| D_\cG(X)} \quad \text{ and } \quad |S_i|, |T_i| \le 4 \theta |X|.
        $$
        Therefore, the desired sets $S_i$ and $T_i$ exist.
    \claimproofend

    If $Y_i=\emptyset$, take $C_i'=\emptyset$, $X_i'=Y_i'=\emptyset$, and $J_i=\emptyset$. Suppose this is not the case. Extend $\phi_i'$ by zero to $V\setminus V_i$, and apply Lemma \ref{lemma:approx_full} with $\phi_i'$, $X$, the subset $Y_i$ (as $Y$), and the family $\{A_j\cap V_i\}_{j\in I}$. Let $X_i'\subseteq X$, $Y_i'\subseteq Y_i$, and $J_i\subseteq I$ be the corresponding sets, and let $C_i' := C(\phi_i', X_i', Y_i', J_i) \cap V_i$. Then
    $$
        |X_i'|+|Y_i'|\le\frac{\alpha|X|}{2t},\qquad |J_i|\le h,
    $$
    and the set $C_i'$ satisfies
    $$
        \Bigl||Y_i\cap A|-|C_i'\cap X_i\cap A|\Bigr|\le\frac{\varepsilon|X|}{t}
        \quad\text{for every $A\in\cA$},
    $$
    and
    $$
        |\phi_i'(Y_i)-\phi_i'(C_i'\cap X_i)|
        \le B\left(\sum_{x\in X_i}\phi_i'(x)^2\right)^{1/2}.
    $$
    Since $\phi_i'(x)\le\deg_{\cG[X]}(x)$ for $x\in X_i$, Cauchy--Schwarz gives
    \begin{align*}
        \sum_{x\in X_i}\phi_i'(x)^2
        &\le\frac{|X|}{(k-1)^2}
            \sum_{x\in X_i}\ \sum_{v\in X\setminus V_i}\deg_{\cG[X]}(x,v)^2\\
        &\le\frac{|X|}{k-1}D_\cG(X)\le |X|D_\cG(X).
    \end{align*}
    Indeed, $(k-1)\deg_{\cG[X]}(x)=\sum_{v\in X\setminus V_i}\deg_{\cG[X]}(x,v)$, and each ordered edge pair in the double sum has at least two common vertices and is counted at most $k-1$ times: it has at most one common vertex in $V_i$.
    Together with \eqref{eq:approx_phi}, this gives
    \begin{align*}
        |\phi_i(Y_i) - \phi_i(C_i' \cap X_i)| &\le |\phi_i(Y_i) - L \phi_i'(Y_i)| + |L\phi_i'(Y_i) - L\phi_i'(C_i' \cap X_i)| \\
        &\quad + |L \phi_i'(C_i' \cap X_i) - \phi_i(C_i' \cap X_i)| \\
        &\le 8L \sqrt{|X|D_\cG(X)} + LB \sqrt{|X|D_\cG(X)} \le \frac{Q\sqrt{|X|D_\cG(X)}}{t}.
    \end{align*}
    Set $C_i = C_{i+1} \cup C_i'$. Edges avoiding $V_i$ contribute equally to the two mixed counts, so 
    $$
        \phi_i(Y_i) - \phi_i(C_i' \cap X_i) = e_{\cG}^{(i+1)}(Y, C_{i+1} \cap X) - e_\cG^{(i)}(Y, C_i \cap X), 
    $$
    and the edge-count bound in \eqref{eq:induction} holds. The occupancy bounds follow by adding the error on $V_i$ to the induction hypothesis for $C_{i+1}$. 

    Set $C := C_1$. By \eqref{eq:induction}, all the required occupancy bounds hold. Note that $e_\cG^{(t+1)}(Y, C_{t+1} \cap X) = e_\cG(Y)$ and $e_\cG^{(1)}(Y, C_1 \cap X) = e_\cG(C \cap X)$. By telescoping and applying \eqref{eq:induction}, we conclude
    $$
        \Bigl| e_{\cG}(Y) - e_\cG(C \cap X) \Bigr| \le \sum_{i = 1}^{t} \left| e_\cG^{(i+1)}(Y, C_{i+1} \cap X) - e_{\cG}^{(i)}(Y, C_i \cap X) \right| \le Q \sqrt{|X|D_\cG(X)}.
    $$

    The recorded sets have total size at most $\alpha|X|$. We claim that $C$ can be obtained from 
    $$
        \mathbf F_X=(T_1,X_1',\ldots,T_t,X_t'), \quad
        \mathbf F_Y=(S_1,Y_1',\ldots,S_t,Y_t'), \quad
        \mathbf J=(J_1,\ldots,J_t).
    $$
    Starting with $C_{t+1}=\emptyset$, recover $C_i$ in descending order of $i$. The function $\phi_i'$ on $V_i$ is determined by $S_i$ and $T_i\cap C_{i+1}$; extend it by zero outside $V_i$. The construction in Lemma \ref{lemma:approx_full}, using $J_i,X_i',Y_i'$ and the fixed indexed family $\{A_j\cap V_i\}_{j\in I}$, then recovers $C_i'$ after intersection with $V_i$. If $Y_i'=\emptyset$, this construction returns the empty set, as required also in the case $Y_i=\emptyset$. Thus $C_i=C_{i+1}\cup C_i'$ is recovered at every step from $\mathbf F_X,\mathbf F_Y,$ and $\mathbf J$.
\end{proof}

\section{Proof of Theorem \ref{thm:main}}
\label{sec:proof_main}

We shall make use of Janson's well-known lower-tail inequality. In particular, for any hypergraph $\cF$ with $\mu_\cF(p)>0$ and any $0\le u\le\mu_\cF(p)$, Janson's inequality \cite[Theorem 1]{riordanwarnke2015} states
\begin{equation} \label{eq:janson}
    \Pr\bigl(e_\cF(V_p)\le\mu_\cF(p) - u \bigr) \le \exp\left(-\frac{u^2}{2\Delta_\cF(p)}\right).
\end{equation}

\begin{proof}[Proof of Theorem \ref{thm:main}]
Write $V=V(\cH)$. Let us first choose constants. Set $\beta=\min\{\varepsilon,1\}/16$ and $t=\lceil k^2/\beta\rceil$. Choose $\alpha\in(0,1]$ sufficiently small that
$$
    \alpha\le\beta/20,\qquad
    \sqrt{5kK\alpha}\le\beta/2,\qquad
    5\alpha\log\left(\frac{e2^{4t+1}}{5\alpha}\right)
        \le\frac{\beta^2}{64K}.
$$
Let $c_0,Q,h$ be as given by Lemma~\ref{lemma:partite} with parameters $k,t,\alpha$ and $\beta$ in place of $\varepsilon$. Choose $\delta>0$ such that
$$
    \delta\le\min\{c_0/4,\alpha/2\},\qquad 2\delta(th+1)\log2\le\frac{\beta^2}{64K}.
$$
Let $\gamma=\beta^2/(2KQ^2)$, and let $T$ be as given by Lemma~\ref{lemma:deletion} applied with $\alpha/2$ and $\gamma$. Enlarge $T$, if necessary, to ensure that all estimates below requiring sufficiently large $Np$ hold: this is possible because $q\ge1/N$ and $Np>T$. In particular, require $Np\ge1/\delta$ and $1+\delta Np\le c_0(1-\alpha)Np$. 

We may assume $\cA\ne\emptyset$, replacing an empty family by $\{V\}$ if necessary. Write $\cA=\{A_j\}_{j\in I}$, with $|I|=|\cA|$. Form a partition $V = V_1 \sqcup \ldots \sqcup V_t$ by assigning each vertex independently and uniformly to one of the parts. For $h \in \cH$, let $I_h$ be the indicator random variable for $|h \cap V_i| \ge 2$ for some $i \in [t]$. Then $\Pr(I_h=1) \le \binom{k}{2}/t$. Thus
$$
    \mathbb{E}\left[ \sum_{h \in \cH} I_h \right] \le \beta e(\cH).
$$
Therefore, with positive probability there exists $V = V_1 \sqcup \ldots \sqcup V_t$ such that all but at most $\beta e(\cH)$ edges $h \in \cH$ intersect each part at most once. Fix one such partition before sampling $X$, and let $\cG \subseteq \cH$ denote the hypergraph consisting of all transversal edges. The bound on $\Delta_\cH(q)$ and $p \ge q$ imply
\begin{equation} \label{eq:cor}
    \Delta_\cH(p)\le K\frac{\mu_\cH(p)^2}{Np} \quad \text{ and } \quad \sum_{v\in V}\deg_\cH(v)^2\le kK\frac{e(\cH)^2}{N}.
\end{equation}

For $\emptyset\ne Y'\subseteq X'\subseteq V$ with $|X'| \ge 1/c_0$ and $|I| < 2^{c_0|X'|}$, fix a choice $\Phi(Y',X')$ of the set $C$ supplied by Lemma \ref{lemma:partite}, applied to $\cG$ and $\cA$ with parameters $\alpha$ and $\beta$. Set $\Phi(Y',X')=\emptyset$ for all other pairs. 

\begin{samepage}
\begin{claim} \label{claim:cleaned_C}
    With probability at least $1 - e^{-cNp}$, we have $\bigl||X|-Np\bigr|\le\alpha Np/2$ and the following holds simultaneously: For every $Y' \subseteq X' \subseteq X$ with $|X'| \ge (1 - \alpha)Np$, the set $C = \Phi(Y', X')$ satisfies 
    \begin{enumerate}[label=\textit{(P\arabic*)}]
        \item \label{prop:C_X_size} $\bigl| |C\cap X\cap A|-p|C\cap A|\bigr|\le\beta Np$ for every $A\in\cA$;
        \item \label{prop:C_X_edges} $e_{\cG[C]}(X') \ge \mu_{\cG[C]}(p)  - \beta e(\cH) p^k$;
        \item \label{prop:C_X_edges_compl} $e_{\cG \setminus \cG[C]}(X') \ge \mu_{\cG \setminus \cG[C]}(p) -  \beta e(\cH) p^k$.
    \end{enumerate}
\end{claim}
\end{samepage}
\claimproofstart
Fix a set $R\subseteq V$, tuples $\mathbf F_X,\mathbf F_Y\in(2^V)^{2t}$, and $\mathbf J=(J_1,\ldots,J_t) \in (2^I)^t$ with $\sum_{i=1}^t|J_i|\le th$. Let $S$ be the union of $R$ and all the sets in $\mathbf F_X$ and $\mathbf F_Y$ (we call such $S$ the \emph{support}). Suppose $|S|\le5\alpha Np$, and condition on $S\subseteq X$. Let $C=C(\mathbf F_X,\mathbf F_Y,\mathbf J)$ be as given by Lemma \ref{lemma:partite}, setting $C=\emptyset$ for records for which this construction is not defined. By the second part of \eqref{eq:cor} and the Cauchy--Schwarz inequality, the number of edges of $\cH$ (and hence $\cG$) meeting $S$ is at most
$$
    \sum_{v\in S}\deg_\cH(v) \le e(\cH)\sqrt{kK|S|/N} \le e(\cH)\sqrt{5kK\alpha}\le\beta e(\cH)/2.
$$
Let $\cG_S=\{h\in\cG[C]:h\cap S=\emptyset\}$ and $\cG'_S=\{h\in\cG\setminus\cG[C]:h\cap S=\emptyset\}$. The previous bound implies
$$
    \mu_{\cG_S}(p) \ge \mu_{\cG[C]}(p) - \beta \mu_{\cH}(p) / 2 \quad \text{ and } \quad \mu_{\cG'_S}(p) \ge \mu_{\cG \setminus \cG[C]}(p)  - \beta \mu_{\cH}(p) / 2.
$$
The set $X\setminus S$ is a $p$-random subset of $V\setminus S$. If the claimed lower bound in \ref{prop:C_X_edges} or \ref{prop:C_X_edges_compl} is nonpositive, it holds deterministically. Otherwise, the corresponding remaining mean exceeds $\beta\mu_\cH(p)/2$, so Janson's inequality applies with $u=\beta\mu_\cH(p)/2$. Since $X\setminus S\subseteq X\setminus R$ and both remaining overlap sums are at most $\Delta_\cH(p)$, each of these properties, with $X'=X\setminus R$, has conditional failure probability at most
$$
    \exp\left(-\frac{\beta^2\mu_\cH(p)^2}{8\Delta_\cH(p)}\right) \le e^{-\beta^2Np/(8K)}.
$$
Moreover, for each $A\in\cA$, the variable $|C\cap A\cap(X\setminus S)|$ is binomial with mean $p|(C\cap A)\setminus S|$. Conditioning on $S\subseteq X$ shifts the mean of $|C\cap A\cap X|$ from $p|C\cap A|$ by at most $|S|\le\beta Np/2$. Chernoff's bound and a union bound over $\cA$ therefore give conditional failure probability at most $2|I|e^{-\beta^2Np/12}$ for \ref{prop:C_X_size}. The probability of any failure, conditional on $S\subseteq X$, is thus at most $4|I|e^{-\beta^2Np/(16K)}$.

For a support $S$ of size $s$, there are at most $2^{(4t+1)s}$ ways to specify membership in $\mathbf F_X$, $\mathbf F_Y$, and $R$. There are at most $(1+t|I|)^{th}\le t^{th}(1+|I|)^{th}$ choices for $\mathbf J$: encode the set of pairs $(i,j)$ with $j\in J_i$ by a list of length $th$ over $[t]\times I$, with an additional blank symbol. The probability that $S\subseteq X$ is $p^s$, so the union bound gives total failure probability at most
\begin{align*}
    &4t^{th}(1+|I|)^{th+1}e^{-\beta^2Np/(16K)}
        \sum_{s=0}^{\lfloor5\alpha Np\rfloor}
        \binom Ns p^s2^{(4t+1)s}\\
    &\quad\le4t^{th}(5\alpha Np+1)(1+|I|)^{th+1}
        \exp\left(-\frac{\beta^2Np}{16K}
            +5\alpha Np\log\frac{e2^{4t+1}}{5\alpha}\right)
        \le e^{-\beta^2Np/(64K)}
\end{align*}
for sufficiently large $Np$. Here we used $\binom Ns p^s\le(eNp/s)^s$ for $s>0$, the monotonicity of $s\log(e2^{4t+1}Np/s)$ on $0<s\le5\alpha Np$, and
$$
    \log(1+|I|)\le(1+\delta Np)\log2\le2\delta Np\log2.
$$
The two choices of $\alpha$ and $\delta$ bound the corresponding entropy terms by $\beta^2Np/(64K)$ each; the remaining constant and polynomial factors are absorbed for large $Np$. Thus the three bounds hold for every such $\mathbf F_X$, $\mathbf F_Y$, $\mathbf J$, and $R$. Finally, Chernoff's bound gives $\bigl||X|-Np\bigr|\le\alpha Np/2$ with probability at least $1 - 2e^{-\alpha^2Np/12}$. Decreasing $c$ absorbs these failure probabilities into $e^{-cNp}$ for sufficiently large $Np$.

Suppose that all these events hold. For any $Y'\subseteq X'\subseteq X$ as in the claim, take $R=X\setminus X'$ and, if $Y'\ne\emptyset$, let $\mathbf F_X$, $\mathbf F_Y$, and $\mathbf J$ be the records used to define $C=\Phi(Y',X')$ in Lemma \ref{lemma:partite}. If $Y'=\emptyset$, take every recorded set and each $J_i$ empty, so that $C=\emptyset$. Every recorded vertex lies in $X$. Then $|R|\le3\alpha Np/2$, $|X'|\le2Np$, and the support has size at most
$$
    |R|+\alpha|X'|\le7\alpha Np/2\le5\alpha Np.
$$
Therefore, the set $C$ satisfies the three properties.
\claimproofend

With probability at least $1-e^{-Np}-e^{-cNp}$, the set $X$ satisfies Lemma \ref{lemma:deletion} (with $\alpha/2$ as $\alpha$) and Claim \ref{claim:cleaned_C}. For sufficiently large $Np$, this is at least $1-e^{-cNp}$ after decreasing $c>0$. Fix such an $X$, and consider any $Y\subseteq X_0\subseteq X$ satisfying
$$
    |X_0|\ge(1-\delta)Np,\qquad
    e_\cH(X_0)\le(1+\varepsilon/4)e(\cH)p^k.
$$
Let $R\subseteq X$ be a set of size $|R|\le\alpha Np/2$ supplied by Lemma \ref{lemma:deletion}, and put $X'=X_0\setminus R$. By monotonicity,
$$
    D_\cG(X')\le D_\cH(X\setminus R)\le\gamma\Delta_\cH(p).
$$
Since $\delta\le\alpha/2$, we have
$$
    |X'|\ge(1-\delta-\alpha/2)Np\ge(1-\alpha)Np,
    \qquad |X'|\le2Np.
$$
Let $Y'=Y\cap X'$ and $C=\Phi(Y',X')$, with $C=\emptyset$ if $Y'=\emptyset$. Claim \ref{claim:cleaned_C} applies to this pair. We argue that $C$ has the desired properties.

Recall that, also trivially when $Y'=\emptyset$,
\begin{equation} \label{eq:recall}
    \begin{gathered}
        \Bigl||Y'\cap A|-|C\cap X'\cap A|\Bigr|\le\beta|X'|\quad(A\in\cA),\\
        \bigl|e_\cG(Y')-e_\cG(C\cap X')\bigr|
        \le Q\sqrt{|X'|D_\cG(X')}
        \le Q\sqrt{2K\gamma}\,e(\cH)p^k
        =\beta e(\cH)p^k.
    \end{gathered}
\end{equation}
We have $|Y\setminus Y'|\le|R|\le\alpha Np/2$ and
$$
    |X\setminus X'|\le|X|-|X_0|+|R|\le(\delta+\alpha)Np.
$$
Together with \ref{prop:C_X_size}, for every $A\in\cA$ this gives
\begin{align*}
    \Bigl||Y\cap A|-p|C\cap A|\Bigr|
    &\le \Bigl||Y\cap A|-|Y'\cap A|\Bigr|
        +\Bigl||Y'\cap A|-|C\cap X'\cap A|\Bigr|\\
    &\quad+\Bigl||C\cap X'\cap A|-|C\cap X\cap A|\Bigr|
        +\Bigl||C\cap X\cap A|-p|C\cap A|\Bigr|\\
    &\le(\delta+3\alpha/2)Np+\beta|X'|+\beta Np
        \le(2\alpha+3\beta)Np\le\varepsilon Np.
\end{align*}
We now bound $e_\cH(Y)$ in terms of $e_\cH(C)p^k$. First, 
$$
    e_\cG(Y') \ge e_\cG(C \cap X') - \beta e(\cH) p^k \stackrel{\ref{prop:C_X_edges}}{\ge} e_\cG(C) p^k - 2\beta  e(\cH) p^k \ge e_\cH(C) p^k - 3 \beta e(\cH) p^k.
$$
Since $e_\cH(Y) \ge e_\cG(Y')$, this proves the desired lower bound. To bound $e_\cH(Y)$ from above, the crucial observation is
$$
    e_\cH(X_0) - e_\cH(Y) \ge e_\cG(X') - e_\cG(Y').
$$
Rearranging and using $e_\cG(Y') \le e_\cG(C \cap X') + \beta e(\cH) p^k$, we get
\begin{align*}
    e_\cH(Y) &\le e_\cG(Y') - e_\cG(X') + e_\cH(X_0)\\
    &\le e_\cG(C \cap X') - e_\cG(X') + e_\cH(X_0) + \beta e(\cH) p^k \\
    &\stackrel{\ref{prop:C_X_edges_compl}}{\le} e_\cG(C) p^k - e(\cG) p^k + e_\cH(X_0) + 2 \beta e(\cH) p^k \\
    &\le e_\cH(C) p^k + (\varepsilon/4+3\beta)e(\cH) p^k.
\end{align*}
Since $3\beta\le\varepsilon$ and $\varepsilon/4+3\beta\le\varepsilon$, both edge-count errors have the required size. The argument holds simultaneously for every such pair $Y\subseteq X_0$, on the same event of probability at least $1-e^{-cNp}$.
\end{proof}

\section{Example application: two-sided counting lemma}
\label{sec:example}

To demonstrate the use of Theorem \ref{thm:main}, we derive the two-sided $H$-counting lemma for subgraphs of random graphs. This is also one of the major results of Allen et al.~\cite{allen26dense}. We derive it from Theorem \ref{thm:main} and a union bound, without using the additional ``subcounts'' supplied by \cite[Theorem 2.4]{allen26dense} (see \cite[Section 15.4]{allen26dense} for a discussion).

For a graph $G$ and disjoint nonempty sets $W_1,W_2 \subseteq V(G)$, let 
$$
    d_G(W_1,W_2) = \frac{e_G(W_1,W_2)}{|W_1||W_2|}
$$
denote the density of the induced bipartite graph $G[W_1,W_2]$. For $\varepsilon > 0$ and $p \in [0,1]$, the pair $(W_1,W_2)$ is \emph{$(\varepsilon,p)$-regular} in $G$ if
$$
    |d_G(W'_1,W'_2)-d_G(W_1,W_2)|\le\varepsilon p
$$
for all $W'_i\subseteq W_i$ of size $|W'_i|\ge\varepsilon|W_i|$ (for $i \in \{1,2\}$). We say that a pair is $\varepsilon$-regular if it is $(\varepsilon, 1)$-regular.

For a graph $H$ on the vertex set $[v]$ and pairwise disjoint subsets $V_1,\ldots,V_v \subseteq V(G)$, let $N_H(G;V_1,\ldots,V_v)$ count the tuples $(x_1,\ldots,x_v)\in V_1\times\cdots\times V_v$ such that $x_ix_j\in E(G)$ for every $ij\in E(H)$. These are the \emph{canonical copies} of $H$. Finally, set
$$
    m_2(H)=\max \left\{\frac{e(F)-1}{v(F)-2}:F\subseteq H,\ v(F)\ge3\right\}.
$$

Conlon, Gowers, Samotij, and Schacht \cite{conlon2014klr} referred to the following result as the ``K\L R conjecture in random graphs'' and proved its two-sided form for strictly $2$-balanced graphs. The result in full generality was proved by Allen et al.~\cite{allen26dense}. For its significance and numerous implications, we refer the reader to \cite{conlon2014klr,gerke_steger_2005}. 

\begin{theorem} \label{thm:counting_klr}
    Let $H$ be a graph on the vertex set $[v]$ with $m_2(H) \ge 1$. For every $\delta,c > 0$ and $\eta > 0$, there exist $\varepsilon = \varepsilon(\delta,c) > 0$ and $B > 0$ such that the following holds. If $p\ge Bn^{-1/m_2(H)}$, then with high probability $\Gamma\sim G(n,p)$ has the following property simultaneously for every $G\subseteq\Gamma$ and all pairwise disjoint $V_1,\ldots,V_v\subseteq[n]$ with $|V_i|\ge\eta n$: if, for every $ij\in E(H)$, the pair $(V_i,V_j)$ is $(\varepsilon,p)$-regular in $G$ and $d_G(V_i,V_j)\ge c p$, then
    $$
        N_H(G;V_1,\ldots,V_v)
        =(1\pm\delta)\left(\prod_{i=1}^v|V_i|\right)
            \prod_{ij\in E(H)}d_G(V_i,V_j).
    $$
\end{theorem}

Note that it is crucial here that $\varepsilon$ does not depend on $\eta$. The striking feature of this result is that regularity of the edge distribution between the vertex classes controls the number of $H$-copies.

When $m_2(H)>1$, the order of the lower bound on $p$ is necessary. Indeed, if $np\to\infty$ and $p=o(n^{-1/m_2(H)})$, then a typical $\Gamma\sim G(n,p)$ can be made $H$-free by removing $o(n^2p)$ edges: delete one edge from each copy of a subgraph attaining $m_2(H)$. Let $G$ be the resulting graph. Chernoff's inequality shows that the pairs between fixed disjoint vertex classes of linear size in $\Gamma$ are $(\varepsilon/2,p)$-regular and have density $(1+o(1))p$. Removing $o(n^2p)$ edges leaves these pairs $(\varepsilon,p)$-regular with densities $(1+o(1))p$. Thus, for any fixed $0<c<1$, they satisfy the regularity and density assumptions, whereas $N_H(G;V_1,\ldots,V_v)=0$.

\begin{proof}[Proof of Theorem \ref{thm:counting_klr}]
    We may assume that $H$ has no isolated vertices, since each isolated vertex contributes a factor $|V_i|$ to both sides of the conclusion. We may also assume $\delta,c\le1$ and $\eta\le1/v$.

    Choose $\varepsilon\in(0,1/2]$ sufficiently small for the usual dense counting lemma, in its form with unequal vertex-class sizes (see, for example, \cite[Lemma 1.3]{conlon2014klr}), to apply with regularity parameter $2\varepsilon$ and additive counting error $\delta c^{e(H)}/2$. That is, given any $D \subseteq K_n$ and disjoint nonempty $V_1, \ldots, V_v \subseteq [n]$ such that, for every $ij \in E(H)$, the pair $(V_i, V_j)$ is $2\varepsilon$-regular in $D$, we have
    \begin{equation} \label{eq:dense_counting}
        N_H(D;V_1,\ldots,V_v) =\left(\prod_{ij\in E(H)}d_D(V_i,V_j) \pm\frac{\delta c^{e(H)}}2\right)\prod_{i=1}^v|V_i|.
    \end{equation}
     
    Let $\cH$ be the $e(H)$-uniform hypergraph on $E(K_n)$ whose hyperedges are the edge sets of copies of $H$. Let $N=\binom n2$ and $q=4n^{-1/m_2(H)}$. By the assumption $m_2(H)\ge1$ we have $q > 1/N$. For $p\ge q$, counting intersecting pairs of copies according to their common vertices and edges gives
    \begin{equation} \label{eq:klr_overlap}
        \frac{\Delta_\cH(p)}{\mu_\cH(p)^2}
        = O_H\left(\sum_{\substack{F\subseteq H\\e(F)\ge1}} n^{-v(F)}p^{-e(F)}\right)
        \le\frac{K}{Np},
    \end{equation}
    where $K\ge1$ depends only on $H$. Indeed, an overlap with $v(F)$ vertices and $e(F)$ edges contributes $O_H(n^{2v-v(F)}p^{2e(H)-e(F)})$, while $e(\cH)=\Theta_H(n^v)$. Moreover,
    $$
        n^{v(F)-2}p^{e(F)-1}\ge n^{v(F)-2}q^{e(F)-1}\ge1
    $$
    by the definition of $m_2(H)$ when $v(F)\ge3$, and directly when $F$ is a single edge. Thus each summand in \eqref{eq:klr_overlap} is at most $1/(n^2p)$. 

    Set
    $$
        \tau := \min\left\{\frac{\varepsilon^3\eta^2}{2},
            \frac{\delta c^{e(H)}}{2(1+e(H)2^{e(H)-1}/\eta^2)}\right\}.
    $$
    For each tuple $\mathbf{V}=(V_1,\ldots,V_v)$ of disjoint subsets of $[n]$ with $|V_i|\ge\eta n$, let $\cH_{\mathbf{V}}\subseteq\cH$ consist of the canonical copies of $H$ on these classes, retaining $V(\cH_{\mathbf{V}}) = V(\cH)$ even if some vertices are isolated. Then
    $$
        e(\cH_{\mathbf{V}})=\prod_{i=1}^v|V_i| \ge\eta^v n^v\ge\lambda e(\cH), 
    $$
    where $\lambda=\eta^v$. Consequently,
    $$
        \Delta_{\cH_{\mathbf{V}}}(q)\le\Delta_\cH(q) \stackrel{\eqref{eq:klr_overlap}}{\le} K\lambda^{-2}\frac{\mu_{\cH_{\mathbf{V}}}(q)^2}{Nq}.
    $$
    Let $\cA$ consist of all sets $E_{K_n}(U,W)$ with $U,W\subseteq[n]$ disjoint. Note that $|\cA|\le3^n$.

    Apply Theorem \ref{thm:main} to each $\cH_{\mathbf{V}}$, with $\tau$ in place of $\varepsilon$, $K\lambda^{-2}$ in place of $K$, and family $\cA$. Let $c_0,\delta_0,T$ be the corresponding constants. Taking $B$ sufficiently large ensures $p>Tq$ and $|\cA|\le3^n\le2^{\delta_0Np}$. As there are at most $(v + 1)^n$ admissible $v$-tuples $\mathbf{V}$ and $Np > Bn/4$, a union bound (after increasing $B$ if necessary) shows that, with high probability, $X \sim V(\cH)_p$ satisfies the conclusion of Theorem \ref{thm:main} for every $\cH_{\mathbf{V}}$ simultaneously. 
    
    Next, we establish events which allow us to take $X' = X$ in the conclusion of Theorem \ref{thm:main}. Set $\theta=\tau\lambda/16$. Chebyshev's inequality gives
    \begin{equation} \label{eq:upper_bound}
        \Pr\bigl(e_\cH(X)\le (1+\theta)e(\cH)p^{e(H)}\bigr) \stackrel{\eqref{eq:klr_overlap}}{\ge} 1 - \frac{K}{\theta^2Np}=1-o(1).
    \end{equation}
    Note that it is crucial here that we are not asking for an upper bound on every $e_{\cH_{\mathbf{V}}}(X)$, but rather only on $e_\cH(X)$. This allows us to avoid taking a large union bound, and we shall shortly see that this is already sufficient. Furthermore, Chernoff's inequality gives $|X|\ge(1-\delta_0)Np$ with high probability.

    Finally, to obtain simultaneous upper bounds on the canonical counts, we combine \eqref{eq:upper_bound} with lower bounds on the complementary counts. For each admissible $\mathbf{V}$, Janson's inequality (see, e.g., \cite[Theorem 1]{riordanwarnke2015}) and \eqref{eq:klr_overlap} give
    \begin{equation} \label{eq:complement_janson}
        \Pr\left(e_{\cH\setminus\cH_{\mathbf{V}}}(X) \ge \bigl(e(\cH)-e(\cH_{\mathbf{V}})-\theta e(\cH)\bigr)p^{e(H)}\right)
        \ge 1 - \exp\left(-\frac{\theta^2Np}{2K}\right).
    \end{equation}
    If the threshold on the left is nonpositive, the event is certain; otherwise, we apply Janson's inequality with deviation $\theta e(\cH)p^{e(H)}$ and use $\Delta_{\cH\setminus\cH_{\mathbf{V}}}(p)\le\Delta_\cH(p)$. Again, by a union bound, taking $B$ sufficiently large ensures that with high probability this event holds simultaneously for every $\mathbf{V}$.
                
    Assume that all the stated events hold for $X \sim V(\cH)_p$. In particular, from \eqref{eq:upper_bound} and \eqref{eq:complement_janson} we conclude
    $$
        e_{\cH_{\mathbf{V}}}(X)
        \le e(\cH_{\mathbf{V}})p^{e(H)}+2\theta e(\cH)p^{e(H)}
        \le(1+\tau/8)e(\cH_{\mathbf{V}})p^{e(H)}.
    $$    
    Together with the lower bound on $|X|$, this allows us to take $X'=X$ in Theorem \ref{thm:main} for every $\cH_{\mathbf{V}}$.
    Let us identify $X$ with $\Gamma \subseteq K_n$, and note that $\Gamma \sim G(n,p)$. Consider some $G \subseteq \Gamma$ and an admissible tuple $\mathbf{V}$ satisfying the regularity and density assumptions, and let $Y \subseteq X$ be the set corresponding to $G$. The conclusion of Theorem \ref{thm:main} supplies a set $C \subseteq V(\cH)$, which corresponds to the graph $D\subseteq K_n$, such that
    \begin{equation} \label{eq:klr_model}
        \begin{gathered}
            \left|N_H(G;V_1,\ldots,V_v)-N_H(D;V_1,\ldots,V_v)p^{e(H)}\right| \le\tau p^{e(H)}\prod_{i=1}^v|V_i|,\\
            |e_G(U,W)-e_D(U,W)p|\le\tau Np \quad\text{for all disjoint $U,W\subseteq[n]$.}
        \end{gathered}
    \end{equation}
    The second property further implies
    \begin{equation} \label{eq:regularity}
        \left|\frac{d_G(U,W)}p-d_D(U,W)\right|
        \le \frac{\tau N}{|U||W|}.
    \end{equation}
    This holds for all disjoint nonempty $U,W\subseteq[n]$.
    
    For $ij\in E(H)$ and $U\subseteq V_i$, $W\subseteq V_j$ with $|U|\ge 2\varepsilon|V_i|$ and $|W|\ge 2\varepsilon|V_j|$, $(\varepsilon, p)$-regularity of $(V_i, V_j)$ in $G$ and \eqref{eq:regularity} give
    $$
        |d_D(U,W)-d_D(V_i,V_j)| \le \varepsilon+\frac{\tau}{4\varepsilon^2\eta^2} +\frac{\tau}{\eta^2}\le2\varepsilon.
    $$
    Thus every relevant pair $(V_i, V_j)$ is $2\varepsilon$-regular with respect to $D$. Moreover, \eqref{eq:regularity} gives
    $$
        \left|\frac{d_G(V_i,V_j)}p-d_D(V_i,V_j)\right|\le\frac{\tau}{\eta^2},
        \qquad \frac{d_G(V_i,V_j)}p\le1+\frac{\tau}{\eta^2}\le2.
    $$
    Since each factor is at most $2$, replacing the factors one at a time gives
    $$
        \left|\prod_{ij\in E(H)}d_D(V_i,V_j)
            -\prod_{ij\in E(H)}\frac{d_G(V_i,V_j)}p\right|
        \le e(H)2^{e(H)-1}\frac{\tau}{\eta^2}.
    $$
    By \eqref{eq:dense_counting} and \eqref{eq:klr_model}, we conclude
    \begin{align*}
        N_H(G; V_1, \ldots, V_v) &= N_H(D; V_1, \ldots, V_v) p^{e(H)} \pm  \tau p^{e(H)} \prod_{i = 1}^v |V_i| \\
        &= \left(\prod_{ij\in E(H)}d_D(V_i,V_j)
            \pm\left(\frac{\delta c^{e(H)}}2+\tau\right)\right)
            p^{e(H)}\prod_{i=1}^v|V_i|\\
        &= \left(\prod_{ij\in E(H)}d_G(V_i,V_j)
            \pm\delta c^{e(H)}p^{e(H)}\right)\prod_{i=1}^v|V_i|,
    \end{align*}
    where the last step follows from the choice of $\tau$, which ensures
    $$
        \frac{\delta c^{e(H)}}2+\tau+
            e(H)2^{e(H)-1}\frac{\tau}{\eta^2}\le\delta c^{e(H)}.
    $$
    Finally, $\prod_{ij\in E(H)}d_G(V_i,V_j)\ge c^{e(H)}p^{e(H)}$, which gives the desired bounds.
\end{proof}

\bibliographystyle{abbrv}
\bibliography{references}

\end{document}